\documentclass[12pt]{amsart}
\usepackage{amssymb}
\usepackage{amsmath}
\usepackage{amsfonts}

\newtheorem{corollary}{Corollary}
\newtheorem{definition}{Definition}
\newtheorem{example}{Example}

\newtheorem{proposition}{Proposition}
\newtheorem{remark}{Remark}

\newtheorem{theorem}{Theorem}
\numberwithin{equation}{section}

\begin{document}
\title{Almost Automorphic Distributions}
\author{Chikh BOUZAR}
\address{Laboratory of Mathematical Analysis and Applications, University of
Oran 1, Ahmed Ben Bella. Oran, Algeria}
\email{ch.bouzar@gmail.com ; bouzar@univ-oran.dz}
\author{Fatima Zohra TCHOUAR}
\address{Laboratory of Mathematical Analysis and Applications, University of
Oran 1, Ahmed Ben Bella. Oran, Algeria}
\email{fatima.tchouar@gmail.com}
\date{}

\begin{abstract}
This work deals with almost automorphy of distributions. We give
characterizations and main properties of these distributions. We also study
the existence of distributional almost automorphic solutions of linear
difference-differential equations.
\end{abstract}

\maketitle

\section{Introduction}

S. Bochner defined explicitly scalar-valued almost automorphic functions and
gave some preliminary results on such functions in \cite{Bochner2} and \cite%
{bochner3}. A general abstract study of almost automorphic functions is done
in the work \cite{veech}. Since then almost automorphy becomes a subject of
interest in mathematical research. It is well known that\ every Bohr almost
periodic function is an almost automorphic one and that the space of almost
automorphic functions is strictly larger then the space of almost periodic
functions.

L.\ Schwartz introduced and studied almost periodic distributions in \cite%
{Schw}.

This work is aimed to introduce and to investigate Bochner almost automorphy
in the setting of Schwartz-Sobolev distributions.

The paper is organised as follows, the second section recalls definitions
and some properties of almost automorphic functions and the spaces $\mathcal{%
D}_{L^{p}}$\ with their topological duals. In the third one, we introduce
smooth almost automorphic functions and give their main properties. Section
four is devoted to the introduction of almost automorphic distributions and
the study of their properties. Section five gives a characterization of
almost automorphic distributions in the spirit of Bochner definition of
almost automorphy of a scalar function. This characterisation will play an
important role in applications. We show that Stepanov almost automorphic
functions are examples of almost automorphic distributions. Finally in the
last section as an application, we study the existence of distributional
almost automorphic solutions of linear difference-differential equations.

\section{Preliminaries}

\ \ \ In this work, we consider functions and distributions defined on the
whole space $\mathbb{R}.$ Let $\mathcal{(C}_{b},\left\Vert .\right\Vert
_{L^{\infty }})$\ denotes the space of continuous, bounded and complex
valued functions on $\mathbb{R}$ endowed with the norm of uniform
convergence on $\mathbb{R},$ it is well-known that $\left( \mathcal{C}%
_{b},\left\Vert .\right\Vert _{L^{\infty }}\right) $ is a Banach algebra.

\begin{definition}
(S. Bochner \cite{Bochner2}) A complex-valued function\ $f$ defined and
continuous on $\mathbb{R}$ is called almost automorphic, if for any sequence
$\left( s_{m}\right) _{m\in \mathbb{N}}\subset \mathbb{R},$ one can extract
a subsequence $\left( s_{m_{k}}\right) _{k}$ such that%
\begin{equation}
g\left( x\right) :=\underset{k\rightarrow +\infty }{\lim }f\left(
x+s_{m_{k}}\right) ~~\text{is well-defined for every }x\in \mathbb{R},
\label{equ5}
\end{equation}%
and%
\begin{equation}
\underset{k\rightarrow +\infty }{\lim }g\left( x-s_{m_{k}}\right) =f\left(
x\right) \text{ for every }x\in \mathbb{R}\text{.}  \label{equ6}
\end{equation}%
The space of almost automorphic functions on $\mathbb{R}$ is denoted by $%
\mathcal{C}_{aa}.$
\end{definition}

\begin{remark}
The function $g$ in the above definition is not necessary continuous but it
is measurable and bounded, so locally integrable$.$
\end{remark}

\begin{remark}
We have the strict inclusion $\mathcal{C}_{ap}$ $\subsetneq $ $\mathcal{C}%
_{aa},~$where $\mathcal{C}_{ap}~$is the space of Bohr almost periodic
functions, see \cite{veech}.
\end{remark}

The following characterization of almost automorphic functions is given in
\cite{bochner3}.

\begin{proposition}
\label{prop1.4}A complex-valued function\ $f$ defined and continuous on $%
\mathbb{R}$ is almost automorphic if and only if for any sequence of real
numbers $\left( s_{m}\right) _{m\in \mathbb{N} },~$one can extract a
subsequence $\left( s_{m_{k}}\right) _{k}$ such that%
\begin{equation*}
\lim_{l\rightarrow +\infty }\lim_{k\rightarrow +\infty }f\left(
x+s_{m_{k}}-s_{m_{l}}\right) =f\left( x\right) ,\text{ }\forall x\in \mathbb{%
R}.
\end{equation*}
\end{proposition}

Some properties of almost automorphic functions are summarized in the
following proposition.

\begin{proposition}
\label{prop1.1-1.2-1.3}

\begin{enumerate}
\item $\mathcal{C}_{aa}$ is a Banach subalgebra of $\mathcal{C}_{b}.$

\item Let$~f\in \mathcal{C}_{aa}~$and$~\lim\limits_{x\rightarrow \infty
}f\left( x\right) =0,~$then~$f=0.$

\item Let$~f\in \mathcal{C}_{aa}$ and $h\in \mathbb{R},$~then $\tau
_{h}f:=f\left( .+h\right) \in \mathcal{C}_{aa}.$

\item If $f\in \mathcal{C}_{aa}$ and $g\in L^{1},$ then the convolution\ $%
f\ast g\in \mathcal{C}_{aa}.$
\end{enumerate}
\end{proposition}

The condition of almost automorphy of a primitive has been established in
\cite{zaki}.

\begin{proposition}
\label{th02} A primitive of an almost automorphic function is almost
automorphic if and only if it is bounded.
\end{proposition}

For more details on the following part see \cite{Schw} and \cite{Sobol}. Let
$\mathcal{C}^{\infty }$ be the space of infinitely differentiable functions
on $\mathbb{R}$~and $p\in \left[ 1,+\infty \right] ,~$the space
\begin{equation*}
\mathcal{D}_{L^{p}}:=\left\{ \varphi \in \mathcal{C}^{\infty }:~\forall i\in
\mathbb{Z}_{+},~\varphi ^{\left( i\right) }\in L^{p}\right\}
\end{equation*}%
endowed with the topology defined by the countable family of norms%
\begin{equation*}
\left\vert \varphi \right\vert _{m,p}=\sum_{i=0}^{m}\left\Vert \varphi
^{\left( i\right) }\right\Vert _{L^{p}},\text{ }m\in \mathbb{Z}_{+},
\end{equation*}%
is a Fr\'{e}chet subalgebra of $\mathcal{C}^{\infty }.~$

The space $\mathcal{\dot{B}~}$is by definition the closure in $\mathcal{D}%
_{L^{\infty }}~$of the space $\mathcal{D}$ of smooth functions with compact
support, it is the space of functions from $\mathcal{D}_{L^{\infty }}~$which
vanish at infinity with all their derivatives. It is clear that the space $%
\mathcal{D}$\ is dense in $\mathcal{D}_{L^{p}},1\leq p<+\infty ,$\ and $%
\mathcal{\dot{B}}.$

The space of $L^{p}-$distributions, denoted by $\mathcal{D}_{L^{p}}^{\prime
},1<p\leq +\infty ,~$is the topological dual space of $\mathcal{D}_{L^{q}},$
where $\frac{1}{p}+\frac{1}{q}=1.$ The topological dual of $\mathcal{\dot{B}}
$ is denoted by $\mathcal{D}_{L^{1}}^{\prime }.$~We have the following
characterizations of $\mathcal{D}_{L^{p}}^{\prime }$,~$1\leq p\leq +\infty .$

\begin{proposition}
Let $T\in $ $\mathcal{D}^{\prime },~$the following statements are equivalent
:

\begin{enumerate}
\item $T\in $ $\mathcal{D}_{L^{p}}^{\prime }.$

\item $T\ast \varphi \in L^{p},\forall \varphi \in \mathcal{D}.$

\item $\exists $ $\left( f_{j}\right) _{j\leq k}\subset
L^{p},T=\sum\limits_{j=0}^{k}f_{j}^{\left( j\right) }.$
\end{enumerate}
\end{proposition}

\begin{remark}
A distribution from the space $\mathcal{D}_{L^{\infty }}^{\prime }~$is
called a bounded distribution.
\end{remark}

\section{Smooth almost automorphic functions}

The space of smooth almost automorphic functions on $\mathbb{R}$ is denoted
and defined as follows
\begin{equation*}
\mathcal{B}_{aa}:=\left\{ \varphi \in \mathcal{C}^{\infty }:~\forall i\in
\mathbb{Z}_{+},~\varphi ^{\left( i\right) }\in \mathcal{C}_{aa}\right\} .
\end{equation*}%
We note that $\mathcal{B}_{aa}\subset \mathcal{D}_{L^{\infty }}~$and that$\
\mathcal{B}_{aa}$ is endowed with the induced topology of $\mathcal{D}%
_{L^{\infty }}.$

Here are some properties of the space $\mathcal{B}_{aa}.$

\begin{proposition}
\label{pr02}

\begin{enumerate}
\item $\mathcal{B}_{aa}=\mathcal{C}_{aa}\cap \mathcal{D}_{L^{\infty }}.$

\item $~\mathcal{B}_{aa}~$is a Fr\'{e}chet subalgebra of $\mathcal{D}%
_{L^{\infty }}.$

\item $\mathcal{B}_{aa}\ast L^{1}\subset \mathcal{B}_{aa}.$
\end{enumerate}
\end{proposition}

\begin{proof}
$~1.~$It is clear that $\mathcal{B}_{aa}\subset \mathcal{C}_{aa}\cap
\mathcal{D}_{L^{\infty }}.~$Now let $f\in \mathcal{C}_{aa}\cap \mathcal{D}%
_{L^{\infty }},~$then by the mean value theorem,~we get%
\begin{equation*}
\left\vert f^{\left( i\right) }\left( x\right) -f^{\left( i\right) }\left(
y\right) \right\vert \leq \underset{z\in \mathbb{R}}{\sup }\left\vert
f^{\left( i+1\right) }\left( z\right) \right\vert \left\vert x-y\right\vert
,~\forall i\in \mathbb{Z}_{+},~\forall x,y\in \mathbb{R},
\end{equation*}%
this gives that$~\forall i\in \mathbb{Z}_{+},~f^{\left( i\right) }~$is
uniformly continuous, and it is known that the derivative of almost
automorphic function is also almost automorphic if and only if it is
uniformly continuous \cite{veech}, so $\forall i\in \mathbb{Z}%
_{+},~f^{\left( i\right) }~\in \mathcal{C}_{aa},$ i.e. $f\in \mathcal{B}%
_{aa} $.

$2.~$As the topology of $\mathcal{B}_{aa}\subset \mathcal{D}_{L^{\infty }}$
is given by the countable family of submultiplicative norms $\left\vert
.\right\vert _{k,\infty },k\in \mathbb{Z}_{+},$ it remains to show the
completeness of $\mathcal{B}_{aa}.$ Let $\left( f_{m}\right) _{m\in \mathbb{N%
}}\subset \mathcal{B}_{aa}$ be a Cauchy sequence, it is clear that$~\forall
i\in \mathbb{Z}_{+},~\left( f_{m}^{\left( i\right) }\right) _{m\in \mathbb{N}%
}$ is a Cauchy sequence in $\mathcal{C}_{aa},$ and since $\left( \mathcal{C}%
_{aa},\left\Vert .\right\Vert _{L^{\infty }}\right) $ is complete, then $%
\forall i\in \mathbb{Z}_{+},~~f_{m}^{\left( i\right) }$ converges uniformly
to $f_{i}\in \mathcal{C}_{aa},~$setting$~f_{0}=f,~$we obtain, due to the
uniform convergence, that$~f\in \mathcal{C}^{\infty }~$and$~\forall i\in
\mathbb{Z}_{+},~f^{\left( i\right) }=f_{i}\in \mathcal{C}_{aa},~$ i.e. $%
\left( f_{m}\right) _{m\in \mathbb{N}}$ converges to $f~$in the topology of$~%
\mathcal{B}_{aa},$ which means that$~\mathcal{B}_{aa}$ is complete.

$3.~$If~ $h\in L^{1}$ and$~f\in \mathcal{B}_{aa},$ then $\left( f\ast
h\right) \in \mathcal{C}^{\infty }~$and $\forall i\in \mathbb{Z}_{+},~\left(
f\ast h\right) ^{\left( i\right) }=f^{\left( i\right) }\ast h\in \mathcal{C}%
_{aa}~$by proposition $\ref{prop1.1-1.2-1.3}-\left( 3\right) .$
\end{proof}

\begin{remark}
We have $\mathcal{B}_{aa}\varsubsetneq \mathcal{C}^{\infty }\cap \mathcal{C}%
_{aa}.$
\end{remark}

The following result is a consequence of the above proposition.

\begin{corollary}
Let $f\in \mathcal{D}_{L^{\infty }},~$then $f\in \mathcal{B}_{aa}$ if and
only if $\forall $ $\varphi \in \mathcal{D},f\ast \varphi \in \mathcal{C}%
_{aa}.$
\end{corollary}

\begin{proof}
The necessity is clear from the above proposition. Let $f\in \mathcal{D}%
_{L^{\infty }},$ take a sequence $\left( \rho _{m}\right) _{m\in \mathbb{N}%
}\subset \mathcal{D}$ such that $\rho _{m}\geq 0,supp\rho _{m}\subset \left[
-\frac{1}{m},\frac{1}{m}\right] $~and$~\int\limits_{\mathbb{R}}\rho
_{m}\left( x\right) dx=1,$ by hypothesis $f\ast \rho _{m}\in \mathcal{C}%
_{aa},\forall m\in \mathbb{N}.$ As
\begin{eqnarray*}
\left\vert f\ast \rho _{m}\left( x\right) -f\left( x\right) \right\vert
&=&\left\vert \int\limits_{\mathbb{R}}\rho _{m}\left( y\right) \left(
f\left( x-y\right) -f\left( x\right) \right) dy\right\vert , \\
&\leq &\sup_{y\in \mathbb{R}}\left\vert f^{\prime }\left( y\right)
\right\vert \int\limits_{\frac{-1}{m}}^{\frac{1}{m}}\rho _{m}\left( y\right)
\left\vert y\right\vert dy,~ \\
&\leq &\frac{1}{m}\sup_{z\in \mathbb{R}}\left\vert f^{\prime }\left(
z\right) \right\vert ,
\end{eqnarray*}%
so $\left( f\ast \rho _{m}\right) _{m\in \mathbb{N}}~$converges uniformly to
$f~$on $\mathbb{R},~$by proposition $\ref{prop1.1-1.2-1.3}-\left( 1\right) ,$
$f\in \mathcal{C}_{aa},~$i.e. $~f\in \mathcal{D}_{L^{\infty }}\cap \mathcal{C%
}_{aa}=\mathcal{B}_{aa}.$
\end{proof}

\section{ Almost automorphic distributions}

\ \ \ \ The goal of this section is the introduction of almost automorphic
distributions and the study of some of their properties.

\begin{theorem}
\label{th03}Let $T\in \mathcal{D}_{L^{\infty }}^{\prime },$ the following
statements are equivalent :

\begin{enumerate}
\item $T\ast \varphi \in \mathcal{C}_{aa},~\forall $ $\varphi \in \mathcal{D}%
.$

\item $\exists $ $\left( f_{j}\right) _{j\leq k}\subset \mathcal{C}%
_{aa},~T=\sum\limits_{j=0}^{k}f_{j}^{\left( j\right) }.$
\end{enumerate}
\end{theorem}

\begin{proof}
$1\Rightarrow 2.~$ Let$~T\in \mathcal{D}_{L^{\infty }}^{\prime },$ then $%
\exists m\in \mathbb{Z}_{+},\exists C>0$~such that%
\begin{equation*}
\left\vert \left\langle T,\psi \right\rangle \right\vert \leq C\left\vert
\psi \right\vert _{m,1},\forall \psi \in \mathcal{D}_{L^{1}}.
\end{equation*}%
For $p\in \mathbb{N},~$we consider a fundamental solution $G~$of the
differential operator $\left( 1-\frac{d^{2}}{dx^{2}}\right) ^{p}$ which
satisfies $G\in \mathcal{C}^{2p-2}$ and with integrable derivatives~of order$%
~\leq 2p-2,$ see \cite{Schw}.~ As $G\in L^{1}$ and $T\in \mathcal{D}%
_{L^{\infty }}^{\prime },$ then $T\ast G$ exists and we have$~$%
\begin{equation*}
T=\left( 1-\frac{d^{2}}{dx^{2}}\right) ^{p}\left( G\ast T\right) .
\end{equation*}%
So $G\in \mathcal{D}_{L^{1}}^{2m+2}:=\left\{ \varphi \in \mathcal{C}%
^{2m+2}:~\forall j\leq 2m+2,~\varphi ^{\left( j\right) }\in L^{1}\right\} ,$
if $p=m+2.$ The space $\mathcal{D}_{L^{1}}^{2m+2}$\ is endowed with the \
norm $\left\vert .\right\vert _{2m+2,1}.$ It is well known that $\mathcal{D}$
is dense in $\mathcal{D}_{L^{1}}^{2m+2},$ as $\mathcal{D\subset D}%
_{L^{1}}\subset \mathcal{D}_{L^{1}}^{2m+2}\subset \mathcal{D}_{L^{1}}^{m},$
then $T~$is extended continuously to the space $\mathcal{D}_{L^{1}}^{2m+2}~$%
and there exists a sequence $\left( \theta _{k}\right) _{k\in \mathbb{N}%
}\subset \mathcal{D}$ such that $\left( \theta _{k}\right) _{k}$ converges
to $G~$in$~\mathcal{D}_{L^{1}}^{2m+2}.$ We also have
\begin{eqnarray*}
\left\vert \left( T\ast \theta _{k}\right) \left( x\right) -\left( T\ast
G\right) \left( x\right) \right\vert &=&\left\vert \left\langle T,\tau _{-x}%
\check{\theta}_{k}-\tau _{-x}\check{G}\right\rangle \right\vert , \\
&\leq &C\left\vert \theta _{k}-G\right\vert _{2m+2,1},
\end{eqnarray*}%
which gives%
\begin{equation*}
\sup_{x\in \mathbb{R}}\left\vert \left( T\ast \theta _{k}\right) \left(
x\right) -\left( T\ast G\right) \left( x\right) \right\vert \underset{%
k\rightarrow +\infty }{\longrightarrow }0,
\end{equation*}%
i.e. the sequence$~\left( T\ast \theta _{k}\right) _{k\in \mathbb{N}}$
converges uniformly to $T\ast G~$on $\mathbb{R},~$by proposition $\ref%
{prop1.1-1.2-1.3}-\left( 1\right) ~$we get $T\ast G\in \mathcal{C}_{aa}.$

$2\Rightarrow 1.~$For $\varphi \in \mathcal{D},T\ast \varphi
=\sum\limits_{j=0}^{k}f_{j}^{\left( j\right) }\ast \varphi
=\sum\limits_{j=0}^{k}f_{j}\ast \varphi ^{\left( j\right) }\in \mathcal{C}%
_{aa}~$due to to proposition$~\ref{prop1.1-1.2-1.3}-\left( 3\right) .$
\end{proof}

\begin{definition}
\label{def2}A distribution$~T\in \mathcal{D}_{L^{\infty }}^{\prime }$ is
said almost automorphic if it satisfies any (hence every) condition of the
above theorem. We denote by $\mathcal{B}_{aa}^{\prime }$ the space of all
almost automorphic distributions on $\mathbb{R}.$
\end{definition}

\begin{remark}
The theorem says that $T\in \mathcal{D}_{L^{\infty }}^{\prime }~$is almost
automorphic if and only if there exist a function $f\in \mathcal{C}_{aa}~$%
and $m\in \mathbb{Z}_{+}~$such that $T=\left( 1-\frac{d^{2}}{dx^{2}}\right)
^{m+2}f,$ where $m$\ is the order of $T.$
\end{remark}

As a consequence of the theorem, we obtain the following result.

\begin{corollary}
Let $\left( T_{m}\right) _{m\in \mathbb{N} }\subset \mathcal{B}_{aa}^{\prime
}$ and $T$ $\in \mathcal{D}_{L^{\infty }}^{\prime }$ such that$~\forall
\varphi \in \mathcal{D},\left( T_{m}\ast \varphi \right) _{m\in \mathbb{N} }$
converges uniformly to $T\ast \varphi ,$~then $T\in \mathcal{B}_{aa}^{\prime
}.$
\end{corollary}

\begin{example}
Every classical almost automorphic function is an almost automorphic
distribution.
\end{example}

\begin{example}
Let $\mathcal{B}_{ap}^{\prime },$ see \cite{Schw}, be the space of almost
periodic Schwartz distributions, then we have the strict inclusion$~\mathcal{%
B}_{ap}^{\prime }\subsetneq \mathcal{B}_{aa}^{\prime }.$
\end{example}

The translate $\tau _{h}T,$ $h\in \mathbb{R},~$of a distribution $T\in $ $%
\mathcal{D}^{\prime }~$is defined by
\begin{equation*}
~\left\langle \tau _{h}T,\varphi \right\rangle =\left\langle T,\tau
_{-h}\varphi \right\rangle ,~\forall \varphi \in \mathcal{D},
\end{equation*}%
where%
\begin{equation*}
\tau _{-h}\varphi \left( x\right) =\varphi \left( x-h\right) ,~\forall x\in
\mathbb{R}.
\end{equation*}%
Denote by $\mathcal{B}_{+,0}^{\prime }$ the space of $T\in \mathcal{D}%
^{\prime }$ such that $\underset{h\rightarrow +\infty }{\lim }\tau _{h}T=0$
in$~\mathcal{D}^{\prime }.~$

The following proposition summarises the main properties of $\mathcal{B}%
_{aa}^{\prime }.$

\begin{proposition}
\label{prop1.7/1.10}

\begin{enumerate}
\item If $T$\ $\in \mathcal{B}_{aa}^{\prime },$ then$~T^{\left( i\right) }$\
$\in \mathcal{B}_{aa}^{\prime },~\forall i\in \mathbb{Z}_{+}.$

\item If $T$\ $\in \mathcal{B}_{aa}^{\prime }\cap \mathcal{B}_{+,0}^{\prime
},$ then $T=0.$

\item If $T$\ $\in \mathcal{B}_{aa}^{\prime },$ then $\tau _{h}T\in \mathcal{%
B}_{aa}^{\prime },~\forall h\in \mathbb{R}.$

\item $\mathcal{B}_{aa}^{\prime }\ast \mathcal{D}_{L^{1}}^{\prime }\subset
\mathcal{B}_{aa}^{\prime }.$

\item $\mathcal{B}_{aa}\times \mathcal{B}_{aa}^{\prime }\subset \mathcal{B}%
_{aa}^{\prime }.$
\end{enumerate}
\end{proposition}

\begin{proof}
$1.~$Obvious.

$2.~$As $T\in \mathcal{B}_{+,0}^{\prime },$ then $\forall \varphi \in
\mathcal{D},\lim_{x\rightarrow +\infty }\left( T\ast \varphi \right) \left(
x\right) =\lim_{x\rightarrow +\infty }\left\langle T,\tau _{-x}\check{\varphi%
}\right\rangle =0$ and as $T$\ $\in \mathcal{B}_{aa}^{\prime },T\ast \varphi
\in \mathcal{C}_{aa},$ by proposition $\ref{prop1.1-1.2-1.3}-\left( 2\right)
,$~we get $T\ast \varphi \equiv 0,~\forall \varphi \in \mathcal{D}.$~On the
other hand $\left\langle T,\varphi \right\rangle =\left( T\ast \check{\varphi%
}\right) \left( 0\right) =0,$ so $T=0.$

$3.~$Let $T\in \mathcal{B}_{aa}^{\prime },~$then $\forall \varphi \in
\mathcal{D}~$and~$\forall h\in \mathbb{R},\tau _{h}T\ast \varphi =\tau
_{h}\left( T\ast \varphi \right) ,$as $T\ast \varphi \in \mathcal{C}_{aa}$
and$~\mathcal{C}_{aa}$ is invariant by translation,~then $\tau _{h}\left(
T\ast \varphi \right) \in \mathcal{C}_{aa},$ so $\tau _{h}T\in \mathcal{B}%
_{aa}^{\prime }.$

$4.~$Let $T\in \mathcal{B}_{aa}^{\prime }$ and $S\in \mathcal{D}%
_{L^{1}}^{\prime },$ there exist $\left( f_{j}\right) _{j\leq k}\subset
\mathcal{C}_{aa}$ such that $T=\sum\limits_{j=0}^{k}f_{j}^{\left( j\right)
}~ $and $\left( g_{j}\right) _{j\leq m}\subset L^{1}$~such that $%
S=\sum\limits_{j=0}^{m}g_{j}^{\left( j\right) },~$these give
\begin{equation*}
\left( T\ast S\right) =\sum\limits_{l=0}^{k}\sum\limits_{j=0}^{m}\left(
f_{l}\ast g_{j}\right) ^{\left( l+j\right) },
\end{equation*}%
by proposition $\ref{prop1.1-1.2-1.3}-\left( 3\right) ,~f_{i}\ast g_{j}\in
\mathcal{C}_{aa},$ since $\mathcal{D}_{L^{\infty }}^{\prime }\ast \mathcal{D}%
_{L^{1}}^{\prime }\subset \mathcal{D}_{L^{\infty }}^{\prime },$ we have $%
T\ast S\in \mathcal{B}_{aa}^{\prime }.$

$5.$~Let$~T\in \mathcal{B}_{aa}^{\prime },$ there exists $\left(
f_{j}\right) _{j\leq k}\subset \mathcal{C}_{aa},$~such that $%
T=\sum\limits_{j=0}^{k}f_{j}^{\left( j\right) }.$ For $\varphi \in \mathcal{B%
}_{aa},$ we have%
\begin{equation*}
\varphi T=\sum\limits_{j=0}^{k}\varphi f_{j}^{\left( j\right)
}=\sum\limits_{j=0}^{k}\sum\limits_{l=0}^{j}\left( -1\right) ^{l}\binom{j}{l}%
\left( \varphi ^{\left( l\right) }~f_{j}\right) ^{\left( j-l\right) },
\end{equation*}%
and since $\mathcal{C}_{aa}$ is an algebra, then $\varphi ^{\left( j\right)
}~f_{i}\in \mathcal{C}_{aa},$ hence $\varphi T\in \mathcal{B}_{aa}^{\prime
}. $
\end{proof}

The next result shows that $\mathcal{B}_{aa}$\ is dense in $\mathcal{B}%
_{aa}^{\prime }.$

\begin{proposition}
A distribution $T\in \mathcal{D}_{L^{\infty }}^{\prime }~$is almost
automorphic if and only if there exists $\left( \varphi _{m}\right) _{m\in
\mathbb{N}}\subset \mathcal{B}_{aa}$ such that $\lim\limits_{m\rightarrow
+\infty }\varphi _{m}=T$\textbf{\ }\ in $\mathcal{D}_{L^{\infty }}^{\prime
}. $
\end{proposition}

\begin{proof}
Suppose that there exists $\left( \varphi _{m}\right) _{m\in \mathbb{N}
}\subset \mathcal{B}_{aa}$ such that $\lim\limits_{m\rightarrow +\infty
}\varphi _{m}=T~\ $in $\mathcal{D}_{L^{\infty }}^{\prime }.$ By the strong
topology of $\mathcal{D}_{L^{\infty }}^{\prime },$ for any bounded subset $B%
\mathcal{\subset D}_{L^{1}}$ we have%
\begin{equation*}
\sup_{\psi \in B}\left\vert \left\langle \varphi _{m}-T,\psi \right\rangle
\right\vert \underset{m\rightarrow +\infty }{\longrightarrow }0.
\end{equation*}%
For a fixed $\varphi \in \mathcal{D},~$the set $B:\mathcal{=}\left\{ \tau
_{-x}\check{\varphi}:x\in \mathbb{R}\right\} $ is bounded in $\mathcal{D}%
_{L^{1}},~$so

\begin{eqnarray*}
\sup_{x\in \mathbb{R}}\left\vert \left( \varphi _{m}\ast \varphi \right)
\left( x\right) -\left( T\ast \varphi \right) \left( x\right) \right\vert
&=&\sup_{x\in \mathbb{R}}\left\vert \left\langle \varphi _{m}-T,\tau _{-x}%
\check{\varphi}\right\rangle \right\vert \\
&=&\sup_{\psi \in B}\left\vert \left\langle \varphi _{m}-T,\psi
\right\rangle \right\vert \underset{m\rightarrow +\infty }{\longrightarrow }%
0,
\end{eqnarray*}%
i.e. the sequence of functions $\left( \varphi _{m}\ast \varphi \right)
_{m\in \mathbb{N} }$ $\subset \mathcal{C}_{aa}~$converges uniformly to $%
\left( T\ast \varphi \right) ,~$by propositions$~\ref{prop1.1-1.2-1.3}%
-\left( 1\right) ,~T\ast \varphi \in \mathcal{C}_{aa},~\forall $ $\varphi
\in \mathcal{D},$ so $T\in \mathcal{B}_{aa}^{\prime }.$

Conversely,$~$let $T\in \mathcal{B}_{aa}^{\prime }$ and take a sequence of
positive test functions $\left( \rho _{m}\right) _{m\in \mathbb{N}}\subset
\mathcal{D}$ such that supp$\rho _{m}\subset \left[ -\frac{1}{m},\frac{1}{m}%
\right] $ and$~\int\limits_{\mathbb{R}}\rho _{m}\left( x\right) dx=1.~$%
Define $\varphi _{m}:=\rho _{m}\ast T\in \mathcal{B}_{aa}.$ We claim that
for any bounded set $U$ of $\mathcal{D}_{L^{1}},~\sup\limits_{\varphi \in
U}\left\vert \left\langle \varphi _{m}-T,\varphi \right\rangle \right\vert
\underset{m\rightarrow +\infty }{\longrightarrow }0$. Indeed, since $T\in
\mathcal{D}_{L^{\infty }}^{\prime },~\exists ~l\in \mathbb{Z}_{+},~\exists
C>0,\left\vert \left\langle T,\varphi \right\rangle \right\vert \leq
C\left\vert \varphi \right\vert _{l,1},~\forall \varphi \in \mathcal{D}%
_{L^{1}},$ and then~%
\begin{equation*}
\left\vert \left\langle \varphi _{m}-T,\varphi \right\rangle \right\vert
=\left\vert \left\langle T,\check{\rho}_{m}\ast \varphi -\varphi
\right\rangle \right\vert \leq C\left\vert \check{\rho}_{m}\ast \varphi
-\varphi \right\vert _{l,1},~\forall \varphi \in \mathcal{D}_{L^{1}}.
\end{equation*}%
On the other hand,%
\begin{equation*}
\left( \check{\rho}_{m}\ast \varphi \right) ^{\left( i\right) }\left(
x\right) -\varphi ^{\left( i\right) }\left( x\right) =\int\limits_{\mathbb{R}%
}\check{\rho}_{m}\left( y\right) \left( \varphi ^{\left( i\right) }\left(
x-y\right) -\varphi ^{\left( i\right) }\left( x\right) \right) dy,
\end{equation*}%
by Minkowski inequality and the mean value theorem we obtain for a $t\in %
\left] 0,1\right[ ,$%
\begin{eqnarray*}
\left\Vert \left( \check{\rho}_{m}\ast \varphi \right) ^{\left( i\right)
}-\varphi ^{\left( i\right) }\right\Vert _{L^{1}} &=&\int\limits_{\mathbb{R}%
}\left\vert \int\limits_{\mathbb{R}}\check{\rho}_{m}\left( y\right) \left(
\varphi ^{\left( i\right) }\left( x-y\right) -\varphi ^{\left( i\right)
}\left( x\right) \right) dy\right\vert dx \\
&\leq &\int\limits_{\left[ \frac{-1}{m},\frac{1}{m}\right] }\check{\rho}%
_{m}\left( y\right) \int\limits_{\mathbb{R}}\left\vert y\right\vert
\left\vert \varphi ^{\left( i+1\right) }\left( x+\left( t-1\right) y\right)
\right\vert dxdy, \\
&\leq &\int\limits_{\left[ \frac{-1}{m},\frac{1}{m}\right] }\left\vert
y\right\vert \check{\rho}_{m}\left( y\right) \int\limits_{\mathbb{R}%
}\left\vert \varphi ^{\left( i+1\right) }\left( z\right) \right\vert dzdy, \\
&\leq &\frac{1}{m}\left\Vert \varphi ^{\left( i+1\right) }\right\Vert
_{L^{1}},~
\end{eqnarray*}%
hence$~$%
\begin{equation*}
\left\vert \left\langle \varphi _{m}-T,\varphi \right\rangle \right\vert
\leq C\left\vert \check{\rho}_{m}\ast \varphi -\varphi \right\vert
_{l,1}\leq \frac{C}{m}\left\vert \varphi \right\vert _{l+1,1},\forall
\varphi \in \mathcal{D}_{L^{1}}.
\end{equation*}%
Take $U$ a bounded~set in $\mathcal{D}_{L^{1}},\ $then $\exists M,~\forall
\varphi \in $ $U,$ $\left\vert \varphi \right\vert _{l+1,1}\leq M,$
consequently we obtain
\begin{equation*}
\sup\limits_{\varphi \in U}\left\vert \left\langle \varphi _{m}-T,\varphi
\right\rangle \right\vert \leq \frac{MC}{m}\underset{m\rightarrow +\infty }{%
\longrightarrow }0,~
\end{equation*}%
which gives the conclusion.
\end{proof}

The following result is an extension to almost automorphic distributions~of
the classical result of Bohl-Bohr on primitives.

\begin{proposition}
\label{prop1.11}A primitive$~$of an almost automorphic distribution is
almost automorphic if and only if it is bounded$.$
\end{proposition}

\begin{proof}
$~~$If $T\in \mathcal{B}_{aa}^{\prime }$ and its primitive $S\in \mathcal{B}%
_{aa}^{\prime },$ it is clear that$~S\in \mathcal{D}_{L^{\infty }}^{\prime
}. $ Conversely, let $S\in \mathcal{D}_{L^{\infty }}^{\prime }$ be a
primitive of $T\in \mathcal{B}_{aa}^{\prime },~$i.e. $S^{\prime }=T,$ so$%
~S\ast \varphi \in L^{\infty },~\forall \varphi \in \mathcal{D},~$and
\begin{equation*}
\left( S\ast \varphi \right) ^{\prime }=S^{\prime }\ast \varphi =T\ast
\varphi \in \mathcal{C}_{aa},~\forall \varphi \in \mathcal{D},
\end{equation*}%
i.e. $\left( S\ast \varphi \right) ~$is a bounded primitive of the almost
automorphic$~$function $\left( T\ast \varphi \right) ,~$thus by proposition $%
\ref{th02},$~$S\ast \varphi \in \mathcal{C}_{aa},\forall \varphi \in
\mathcal{D},$ so $S\in \mathcal{B}_{aa}^{\prime }.$
\end{proof}

\section{Bochner almost automorphy of distributions}

This section gives a characterization of almost automorphic distributions in
the spirit of the topological definition of almost automorphy of a scalar
function given by S. Bochner. We also show that almost automorphic functions
in the sense of Stepanov are almost automorphic distributions.

\begin{theorem}
\label{prop1.8}Let $T\in \mathcal{D}_{L^{\infty }}^{\prime },~$the following
propositions are equivalent :

\begin{enumerate}
\item There exists $\left( f_{j}\right) _{j\leq n}\subset \mathcal{C}_{aa}$
such that~$T=\sum\limits_{j=0}^{n}f_{j}^{\left( j\right) }.$

\item For every sequence $\left( s_{m}\right) _{m\in \mathbb{N} }\subset
\mathbb{R},$~there exists a subsequence $\left( s_{m_{k}}\right) _{k}~$such
that%
\begin{equation*}
S:=\lim_{k\rightarrow +\infty }\tau _{s_{m_{k}}}T~~\text{exists in }\mathcal{%
D}^{\prime },
\end{equation*}%
and%
\begin{equation*}
\lim_{l\rightarrow +\infty }\tau _{-s_{m_{l}}}S=T~\text{in }\mathcal{D}%
^{\prime }.
\end{equation*}

\item $T\ast \varphi \in \mathcal{C}_{aa},\forall \varphi \in \mathcal{D}.$
\end{enumerate}
\end{theorem}

\begin{proof}
$1\Rightarrow 2.~$Let $\left( f_{j}\right) _{j\leq n}\subset \mathcal{C}%
_{aa} $ and~$T=\sum\limits_{j=0}^{n}f_{j}^{\left( j\right) },$ for every
sequence $\left( s_{m}\right) _{m\in \mathbb{N}}\subset \mathbb{R},$~there
is a subsequence $\left( s_{m_{k_{1}}}\right) _{k_{1}}$ such that$~\forall
x\in \mathbb{R},$
\begin{equation*}
\lim\limits_{k_{1}\rightarrow +\infty }f_{1}\left( x+s_{m_{k_{1}}}\right)
=g_{1}\left( x\right) ~\text{and }\lim\limits_{k_{1}\rightarrow +\infty
}g_{1}\left( x-s_{m_{k_{1}}}\right) =f_{1}\left( x\right) ,\text{ }
\end{equation*}%
since $f_{2}$ is also an almost automorphic,~one extracts a subsequence $%
\left( s_{m_{k_{2}}}\right) _{k_{2}}$ of $\left( s_{m_{k_{1}}}\right)
_{k_{1}}$ such that $\forall x\in \mathbb{R},$
\begin{equation*}
\lim\limits_{_{k_{2}}\rightarrow +\infty }f_{2}\left( x+s_{m_{k_{2}}}\right)
=g_{2}\left( x\right) ~\text{and }\lim\limits_{_{k_{2}}\rightarrow +\infty
}g_{2}\left( x-s_{m_{k_{2}}}\right) =f_{2}\left( x\right) ,\text{ }
\end{equation*}%
and also,
\begin{equation*}
\lim\limits_{_{k_{2}}\rightarrow +\infty }f_{1}\left( x+s_{m_{k_{2}}}\right)
=g_{1}\left( x\right) \text{ and }\lim\limits_{_{k_{2}}\rightarrow +\infty
}g_{1}\left( x-s_{m_{k_{2}}}\right) =f_{1}\left( x\right) ,
\end{equation*}%
hence by iterating this process, we can extract a subsequence $\left(
s_{m_{k_{n}}}\right) _{_{k_{n}}}$ of $\left( s_{m_{k_{n-1}}}\right)
_{k_{n-1}},$~which is also a subsequence of $\left( s_{m}\right) _{m\in
\mathbb{N}}~$such that~$\forall x\in \mathbb{R},$%
\begin{equation*}
\lim\limits_{_{k_{n}}\rightarrow +\infty }f_{i}\left( x+s_{m_{k_{n}}}\right)
=g_{i}\left( x\right) ~\text{and }\lim\limits_{_{_{k_{n}}}\rightarrow
+\infty }g_{i}\left( x-s_{m_{k_{n}}}\right) =f_{i}\left( x\right)
,~i=1,...,n.
\end{equation*}%
On the other hand, the function$~f_{i}$ is bounded and $g_{i}\in
L_{loc}^{1}, $~so $\forall \varphi \in \mathcal{D},$
\begin{equation*}
\left\langle \tau _{s_{m_{k_{n}}}}T,\varphi \right\rangle
=\sum\limits_{i=0}^{n}\left( -1\right) ^{i}\int\limits_{\mathbb{R}%
}f_{i}\left( x+s_{m_{k_{n}}}\right) \varphi ^{\left( i\right) }\left(
x\right) dx,
\end{equation*}%
then by Lebesgue's dominated convergence theorem, we obtain
\begin{eqnarray*}
\lim_{_{k_{n}}\rightarrow +\infty }\left\langle \tau
_{s_{m_{k_{n}}}}T,\varphi \right\rangle &=&\sum\limits_{i=0}^{n}\left(
-1\right) ^{i}\int_{\mathbb{R}}g_{i}\left( x\right) \varphi ^{\left(
i\right) }\left( x\right) dx, \\
&=&\left\langle \sum\limits_{i=0}^{n}g_{i}^{\left( i\right) },\varphi
\right\rangle , \\
&=&\left\langle S,\varphi \right\rangle ,
\end{eqnarray*}%
where $S:=\sum\limits_{i=0}^{n}g_{i}^{\left( i\right) }\in \mathcal{D}%
^{\prime }.~$In the same way, since the function $g_{i}$ is measurable and
bounded, we obtain $\forall \varphi \in \mathcal{D},$%
\begin{eqnarray*}
\lim_{_{k_{n}}\rightarrow +\infty }\left\langle \tau
_{-s_{m_{k_{n}}}}S,\varphi \right\rangle &=&\sum\limits_{i=0}^{n}\left(
-1\right) ^{i}\int_{\mathbb{R}}\lim_{_{k_{n}}\rightarrow +\infty
}g_{i}\left( x-s_{m_{k_{n}}}\right) \varphi ^{\left( i\right) }\left(
x\right) dx, \\
&=&\left\langle \sum\limits_{i=0}^{n}f_{i}^{\left( i\right) },\varphi
\right\rangle , \\
&=&\left\langle T,\varphi \right\rangle ,
\end{eqnarray*}%
i.e.%
\begin{equation*}
\lim_{_{k_{n}}\rightarrow +\infty }\tau _{-s_{m_{k_{n}}}}S=T\text{ in }%
\mathcal{D}^{\prime }.
\end{equation*}

$2\Rightarrow 3.~~$Let $\varphi \in \mathcal{D},$ for any sequence $\left(
s_{m}\right) _{m\in \mathbb{N} }\subset \mathbb{R},$ one can extract a
subsequence $\left( s_{m_{k}}\right) _{k\in \mathbb{N} }$ such that$~\forall
x\in \mathbb{R},$ we have
\begin{eqnarray*}
\lim_{k\rightarrow +\infty }\left( T\ast \varphi \right) \left(
x+s_{m_{k}}\right) &=&\lim_{k\rightarrow +\infty }\left\langle \tau
_{s_{m_{k}}}T,\tau _{-x}\check{\varphi}\right\rangle , \\
&=&\left\langle S,\tau _{-x}\check{\varphi}\right\rangle , \\
&=&\left( S\ast \varphi \right) \left( x\right) ,
\end{eqnarray*}%
and in the same way $\forall x\in \mathbb{R},$

\begin{eqnarray*}
\lim\limits_{k\rightarrow +\infty }\left( S\ast \varphi \right) \left(
x-s_{m_{k}}\right) &=&\lim_{k\rightarrow +\infty }\left\langle \tau
_{-s_{m_{k}}}S,\tau _{-x}\check{\varphi}\right\rangle , \\
&=&\left\langle T,\tau _{-x}\check{\varphi}\right\rangle , \\
&=&\left( T\ast \varphi \right) \left( x\right) ,~
\end{eqnarray*}%
i.e. $T\ast \varphi \in \mathcal{C}_{aa},~\forall \varphi \in \mathcal{D}$.

$3\Rightarrow 1.~$Follows from theorem \ref{th03}.
\end{proof}

The equivalent properties established in the last theorem lead to a Bochner
type definition of an almost automorphic distribution, i.e. a distribution $%
T\in \mathcal{D}_{L^{\infty }}^{\prime }~$is an almost automorphic
distribution if for every sequence $\left( s_{m}\right) _{m\in \mathbb{N}%
}\subset \mathbb{R},$ there is a subsequence $\left( s_{m_{k}}\right) _{k}~$%
such that%
\begin{equation*}
S:=\lim_{k\rightarrow +\infty }\tau _{s_{m_{k}}}T~~\text{exists in }\mathcal{%
D}^{\prime },
\end{equation*}%
and%
\begin{equation*}
\lim_{k\rightarrow +\infty }\tau _{-s_{m_{k}}}S=T~\text{in }\mathcal{D}%
^{\prime }.
\end{equation*}%
We have also an extension of the proposition $\ref{prop1.4}$ to
distributions.

\begin{proposition}
\label{prop1.9}Let $T\in \mathcal{D}_{L^{\infty }}^{\prime },$ then $T\in
\mathcal{B}_{aa}^{\prime }$ if and only if for every sequence $\left(
s_{m}\right) _{m\in \mathbb{N} }\subset \mathbb{R},$~there exists a
subsequence $\left( s_{m_{k}}\right) _{k}~~$such that
\begin{equation}
\lim_{l\rightarrow +\infty }\lim_{k\rightarrow +\infty }\tau
_{-s_{m_{l}}}\tau _{s_{m_{k}}}T=T~\text{in}~\mathcal{D}^{\prime }.
\label{equa4}
\end{equation}
\end{proposition}

\begin{proof}
If $T\in \mathcal{B}_{aa}^{\prime },$ by the above theorem, for every
sequence $\left( s_{m}\right) _{m\in \mathbb{N}}\subset \mathbb{R},$ there
is a subsequence $\left( s_{m_{k}}\right) _{k}~$such that%
\begin{equation*}
S:=\lim_{k\rightarrow +\infty }\tau _{s_{m_{k}}}T~~\text{and}%
\lim_{k\rightarrow +\infty }\tau _{-s_{m_{k}}}S=T~\text{hold in }\mathcal{D}%
^{\prime }.
\end{equation*}%
Let $\varphi \in \mathcal{D},~$%
\begin{eqnarray*}
\lim_{k\rightarrow +\infty }\left\langle \tau _{-s_{ml}}\tau
_{s_{m_{k}}}T,\varphi \right\rangle &=&\lim_{k\rightarrow +\infty
}\left\langle \tau _{s_{m_{k}}}T,\tau _{s_{m_{l}}}\varphi \right\rangle , \\
&=&\left\langle S,\tau _{s_{m_{l}}}\varphi \right\rangle ,
\end{eqnarray*}%
and%
\begin{eqnarray*}
\lim_{l\rightarrow +\infty }\lim_{k\rightarrow +\infty }\left\langle \tau
_{-s_{m_{l}}}\tau _{s_{m_{k}}}T,\varphi \right\rangle &=&\lim_{l\rightarrow
+\infty }\left\langle \tau _{-s_{m_{l}}}S,\varphi \right\rangle , \\
&=&\left\langle T,\varphi \right\rangle ,
\end{eqnarray*}%
i.e. $\lim_{l\rightarrow +\infty }\lim_{k\rightarrow +\infty }\tau
_{-s_{m_{l}}}\tau _{s_{m_{k}}}T=T~$in$~\mathcal{D}^{\prime }.$

Conversely, let a sequence $\left( s_{m}\right) _{m\in \mathbb{N}}\subset
\mathbb{R}$ with a subsequence $\left( s_{m_{k}}\right) _{k}$ such that $%
\left( \ref{equa4}\right) $ holds, then $\forall \varphi \in \mathcal{D}$
and $\forall x\in \mathbb{R},$ we have
\begin{eqnarray*}
\lim_{l\rightarrow +\infty }\lim_{k\rightarrow +\infty }\left( T\ast \varphi
\right) \left( x+s_{m_{k}}-s_{m_{l}}\right) &=&\lim_{l\rightarrow +\infty
}\lim_{k\rightarrow +\infty }\left\langle \tau _{-s_{m_{l}}}\tau
_{s_{m_{k}}}T,\tau _{-x}\check{\varphi}\right\rangle , \\
&=&\left\langle T,\tau _{-x}\check{\varphi}\right\rangle , \\
&=&\left( T\ast \varphi \right) \left( x\right) ,
\end{eqnarray*}%
hence by proposition \ref{prop1.4}, $T\ast \varphi \in \mathcal{C}%
_{aa},~\forall \varphi \in \mathcal{D},$ i.e. $T\in \mathcal{B}_{aa}^{\prime
}.$
\end{proof}

W. Stepanov introduced the class of almost periodic functions for which only
local integrability in the sense of Lebesgue is required, this class was
extended to the so called Stepanov almost automorphic functions, which is a
more general concept than Bochner almost automorphy. We recall this
definition, see \cite{cazarino}.

\begin{definition}
\label{def1}A function $f$ $\in L_{loc}^{p},~1\leq p<\infty ,~$is said $%
S^{p}-$almost automorphic function$~(S^{p}-a.a.)~$if for every sequence $%
\left( s_{m}\right) _{m\in \mathbb{N} }\subset \mathbb{R},$ there is a
subsequence $\left( s_{m_{k}}\right) _{k\in \mathbb{N} }$ and a function $%
g\in L_{loc}^{p}$ such that
\begin{equation}
\underset{k\rightarrow +\infty }{\lim }\left( \int\limits_{0}^{1}\left\vert
f\left( x+s_{m_{k}}+t\right) -g\left( x+t\right) \right\vert ^{p}dt\right) ^{%
\frac{1}{p}}=0,~\text{for every }x\in \mathbb{R},  \label{equ2}
\end{equation}%
and
\begin{equation}
\underset{k\rightarrow +\infty }{\lim }\left( \int\limits_{0}^{1}\left\vert
g\left( x-s_{m_{k}}+t\right) -f\left( x+t\right) \right\vert ^{p}dt\right) ^{%
\frac{1}{p}}=0,~\text{for every }x\in \mathbb{R}.  \label{equ3}
\end{equation}
\end{definition}

It is well known that if$~1\leq p\leq q<\infty ,$ every $S^{q}-a.a.$
function is an $S^{p}-a.a.~$function, so the space of $S^{1}-a.a.~$functions
contains all the $S^{p}-a.a.~$functions.

\begin{proposition}
If $f~$is $S^{1}-$almost automorphic, then the associated distribution $%
T_{f}~$is almost automorphic.
\end{proposition}

\begin{proof}
Recall that if $g\in L_{loc}^{1},\left\langle T_{g},\varphi \right\rangle
:=\int\limits_{\mathbb{R}}g\left( t\right) \varphi \left( t\right)
dt,\forall \varphi \in \mathcal{D}.$~Let $f$ \ be an $S^{1}-$almost
automorphic, then for every sequence $\left( s_{m}\right) _{m\in \mathbb{N}%
}\subset \mathbb{R},$ there exist a subsequence $\left( s_{m_{k}}\right)
_{k\in \mathbb{N}}$ and a function $g\in L_{loc}^{1}$,~such that $\left( \ref%
{equ2}\right) $ and $\left( \ref{equ3}\right) $ are satisfied. Let $T_{g}$
be the associated distribution to $g$ and let us show that
\begin{equation*}
T_{g}=\lim_{k\rightarrow +\infty }\tau _{s_{m_{k}}}T_{f}\text{ and }%
\lim_{k\rightarrow +\infty }\tau _{-s_{m_{k}}}T_{g}=T_{f}~\text{ exist in }%
\mathcal{D}^{\prime }.
\end{equation*}%
Indeed, for $\varphi \in \mathcal{D}$ with $supp\varphi \subset \left[ a,b%
\right] ~$we can assume that $a$ and $b$ are integers, so
\begin{eqnarray*}
\left\vert \left\langle \tau _{s_{m_{k}}}T_{f},\varphi \right\rangle
-\left\langle T_{g},\varphi \right\rangle \right\vert &=&\left\vert
\int\limits_{a}^{b}\left( f\left( s_{m_{k}}+t\right) -g\left( t\right)
\right) \varphi \left( t\right) dt\right\vert , \\
&\leq &\sup_{x\in \mathbb{R}}\left\vert \varphi \left( x\right) \right\vert
\sum_{i=a}^{b-1}\int\limits_{i}^{i+1}\left\vert f\left( s_{m_{k}}+t\right)
-g\left( t\right) \right\vert dt, \\
&\leq &\sup_{x\in \mathbb{R}}\left\vert \varphi \left( x\right) \right\vert
\sum_{i=a}^{b-1}\int\limits_{0}^{1}\left\vert f\left( i+s_{m_{k}}+t\right)
-g\left( i+t\right) \right\vert dt,
\end{eqnarray*}%
consequently, $\lim\limits_{k\rightarrow \infty }\left\vert \left\langle
\tau _{s_{m_{k}}}T_{f},\varphi \right\rangle -\left\langle T_{g},\varphi
\right\rangle \right\vert =0.~$In the same way, we obtain that
\begin{equation*}
\left\vert \left\langle \tau _{-s_{m_{k}}}T_{g},\varphi \right\rangle
-\left\langle T_{f},\varphi \right\rangle \right\vert \underset{k\rightarrow
+\infty }{\longrightarrow }0,
\end{equation*}%
hence the conclusion.
\end{proof}

\section{Applications to difference-differential equations}

We~consider first linear difference-differential operators%
\begin{equation}
L_{h}=\sum\limits_{i=0}^{p}\sum\limits_{j=0}^{q}a_{ij}\frac{d^{i}}{dx^{i}}%
\tau _{h_{j}},  \label{equa8}
\end{equation}%
where$~\left( a_{ij}\right) _{i\leq p,~j\leq q}~$are complex numbers and $%
h=\left( h_{j}\right) _{j\leq q}\subset \mathbb{R}^{q}.$

\begin{remark}
It is easy to see that $L_{h}T\in \mathcal{B}_{aa}^{\prime },\forall T~\in
\mathcal{B}_{aa}^{\prime }.$
\end{remark}

Let $p\in \mathbb{Z}_{+}~$and denote by $\mathcal{C}_{ub}^{p}$~the space of
functions$~\varphi $ of class $\mathcal{C}^{p}$ such that for every $j\leq
p,~\varphi ^{\left( j\right) }~$is uniformly continuous and bounded on $%
\mathbb{R}.$

The following result is an extension of theorem $4$ of \cite{bochner3}.

\begin{theorem}
\label{th04}If every solution $f\in \mathcal{C}_{ub}^{p}$ of the homogeneous
equation$~$%
\begin{equation}
L_{h}f=0  \label{equa1}
\end{equation}%
is an almost automorphic function, then$~$any solution$~T~\in \mathcal{D}%
_{L^{\infty }}^{\prime }~$of the inhomogeneous equation$~$
\begin{equation}
L_{h}T=S\in \mathcal{B}_{aa}^{\prime }  \label{equa6}
\end{equation}%
is an almost automorphic distribution$.$
\end{theorem}

\begin{proof}
~Let $T\in \mathcal{D}_{L^{\infty }}^{\prime }$ be a solution of equation $%
\left( \ref{equa6}\right) ,~$then for$~$every $\varphi \in \mathcal{D},$ we
have%
\begin{equation*}
L_{h}\left( T\ast \varphi \right) =S\ast \varphi .
\end{equation*}%
On the other hand, $\forall j\in \mathbb{Z}_{+},~\left( T\ast \varphi
\right) ^{\left( j\right) }=T\ast \varphi ^{\left( j\right) }\in \mathcal{D}%
_{L^{\infty }},$ it is clear that $T\ast \varphi ^{\left( j\right) }~$is a
uniformly continuous function on $\mathbb{R},$ so $T\ast \varphi \in
\mathcal{C}_{ub}^{\infty }$ and it is a solution of equation $\left( \ref%
{equa6}\right) $ with a second member $S\ast \varphi \in \mathcal{C}_{aa}$
instead of $S,~$this together with the assumption on the homogeneous
equation $\left( \ref{equa1}\right) ~$imply that theorem $4$ of \cite%
{bochner3} is applicable, consequently $T\ast \varphi \in \mathcal{C}%
_{aa},~\forall \varphi \in \mathcal{D},$ which gives $T\in \mathcal{B}%
_{aa}^{\prime }.$
\end{proof}

The following consequences of the theorem recapture the Bohl-Bohr result and
the invariance by translation of the space of almost automorphic
distributions.

\begin{corollary}
\begin{enumerate}
\item If$~$ $T\in \mathcal{D}_{L^{\infty }}^{\prime }$ and $\dfrac{dT}{dx}%
\in \mathcal{B}_{aa}^{\prime },$ then $T\in \mathcal{B}_{aa}^{\prime }.$

\item If $T\in \mathcal{B}_{aa}^{\prime },$ then $\tau _{h}T\in \mathcal{B}%
_{aa}^{\prime },\forall h\in \mathbb{R}.$
\end{enumerate}
\end{corollary}

A particular case of the operators $\left( \ref{equa8}\right) $ are linear
ordinary differential operators $L=\sum\limits_{i=0}^{p}a_{i}\frac{d^{i}}{%
dx^{i}},~$these operators can be tackled in the more general situation of
systems
\begin{equation}
U^{\prime }=AU+S,  \label{equ1}
\end{equation}%
where $A=\left( a_{ij}\right) _{1\leq i,j\leq p}$ is a given square-matrix
of complex numbers, also the vector distribution$~S=\left( S_{i}\right)
_{1\leq i\leq p}\in \left( \mathcal{D}^{\prime }\right) ^{p},$ and $U=\left(
U_{i}\right) _{1\leq i\leq p\text{ }}$is the unknown vector distribution.

\begin{theorem}
\label{thm1}If all$~S_{i},~1\leq i\leq p~$are almost automorphic
distributions and the matrix $A$ has no eigenvalues with real part zero,
then the equation $\left( \ref{equ1}\right) $ admits a unique solution $U\in
\left( \mathcal{D}_{L^{\infty }}^{\prime }\right) ^{p}~$which is, in fact,
an almost automorphic vector distribution.
\end{theorem}

\begin{proof}
Consider the equation $\left( \ref{equ1}\right) $ and let $\varphi \in
\mathcal{D},$ then%
\begin{equation*}
\left( U\ast \varphi \right) ^{\prime }=A\left( U\ast \varphi \right)
+\left( S\ast \varphi \right) ,
\end{equation*}%
where $U\ast \varphi =\left( U_{i}\ast \varphi \right) _{1\leq i\leq p}$ and
$\left( S\ast \varphi \right) =\left( S_{i}\ast \varphi \right) _{1\leq
i\leq p},$ which gives the following system of equations
\begin{equation*}
v^{\prime }=Av+g
\end{equation*}%
with $g=S\ast \varphi \in \left( \mathcal{C}_{aa}\right) ^{p}$ and $v=U\ast
\varphi \in \left( \mathcal{C}^{\infty }\right) ^{p},$ consequently we apply
theorem $2$ of \cite{zaidman} to obtain that there exists a unique $v\in
\left( \mathcal{C}_{aa}\right) ^{p},$ so $U_{i}\ast \varphi \in \mathcal{C}%
_{aa},$ $\forall \varphi \in \mathcal{D},\forall i=1,...,p,~$i.e. $U\in (%
\mathcal{B}_{aa}^{\prime })^{p}.$
\end{proof}

We conclude with the following consequence of the last theorem.

\begin{corollary}
If the polynomial$~\sum\limits_{i=0}^{p}a_{i}\lambda ^{i}~$has no roots with
real part zero, then any solution $T\in \mathcal{D}_{L^{\infty }}~$of the
inhomogeneous equation$~LT=S\in \mathcal{B}_{aa}^{\prime }$ is an almost
automorphic distribution$.$
\end{corollary}

\end{document}